\documentclass[11pt]{amsart}
\usepackage[left=1in,right=1in,top=1in,bottom=1in]{geometry}

\usepackage[all,cmtip]{xy}
\usepackage{amssymb}
\usepackage[english]{babel}
\usepackage{mathrsfs}

\newtheorem{theorem}{Theorem}[section]
\newtheorem{lemma}[theorem]{Lemma}

\newtheorem{proposition}[theorem]{Proposition}
\newtheorem{corollary}[theorem]{Corollary}
\newtheorem{fact}[theorem]{Fact}

\theoremstyle{definition}
\newtheorem{definition}[theorem]{Definition}

\newtheorem{remark}[theorem]{Remark}

\newtheorem{question}[theorem]{Question}

\newtheorem{claim}{Claim}

\def\N{\mathbb{N}}

\def\T{\mathbb{T}}

\def\grp#1{\langle{#1}\rangle}
\def\dual#1{{{#1}^\wedge}}
\def\seconddual#1{{#1}^{\wedge\wedge}}

\def\cont{\mathfrak{c}}

\begin{document}

\title{The impact of the Bohr topology on selective pseudocompactness} 

\author[D. Shakhmatov]{Dmitri Shakhmatov}
\address{Division of Mathematics, Physics and Earth Sciences\\
Graduate School of Science and Engineering\\
Ehime University, Matsuyama 790-8577, Japan}
\email{dmitri.shakhmatov@ehime-u.ac.jp}
\thanks{The first listed author was partially supported by the Grant-in-Aid for Scientific Research~(C) No.~26400091 of the Japan Society for the Promotion of Science (JSPS)}

\author[V. Ya\~nez]{V\'{\i}ctor Hugo Ya\~nez}
\address{Master's Course, Graduate School of Science and Engineering\\
Ehime University, 
Matsuyama 790-8577, Japan}
\email{victor\textunderscore yanez@comunidad.unam.mx}
\thanks{This paper was written as part of the second listed author's Master's Program at the Graduate School of Science and Engineering of Ehime University. The second listed author was partially supported by the 2016-2017 fiscal year grant of the Matsuyama Saibikai.}

\begin{abstract}
Recall that a space $X$ is {\em selectively pseudocompact\/} if for every sequence $\{U_n: n \in \N\}$ of non-empty subsets of $X$ one can choose a point $x_n \in U_n$ for all $n \in \N$ such that the resulting sequence $\{x_n: n \in \N\}$ has 
an accumulation point in $X$.
This notion
was introduced under the name strong pseudocompactness by Garc\'{\i}a-Ferreira and Ortiz-Castillo;
the present name is due to Dorantes-Aldama and the first author.
In 2015, Garc\'{\i}a-Ferreira and Tomita 
constructed
a pseudocompact Boolean group that is not selectively pseudocompact. 
We prove that if the subgroup topology on every 
countable subgroup $H$ of an infinite Boolean topological group $G$ is finer than its maximal precompact topology (the so-called Bohr topology of $H$), then $G$ is not selectively pseudocompact,
and from this result
we deduce that many known examples in the literature of pseudocompact Boolean groups {\em automatically\/} fail to be selectively pseudocompact.
We also show that,
under the Singular Cardinal Hypothesis, {\em every\/} infinite pseudocompact Boolean group admits a pseudocompact reflexive group topology which is not selectively pseudocompact. 
\end{abstract}

\maketitle

As usual, $\N$ denotes the set of natural numbers and $\cont$ denotes the cardinality of the continuum.

A group $G$ is {\em Boolean\/} if $x^2=e$ for every $x\in G$, where $e$ is the identity element of $G$. It is known (and easy to see) that all Boolean groups are abelian.

For a subset $X$ of a group $G$, the symbol $\grp{X}$ denotes the subgroup of $G$ generated  by $X$; that is, the smallest subgroup of $G$ containing $X$.

A topological group $G$ is {\em precompact\/}, or {\em totally bounded\/}, provided that for every 
open neighbourhood $U$ of the identity of $G$ one can find a finite subset $F$ of $G$ such that $G=FU$. It is well known that a 
topological group $G$ is precompact if and only if it is a subgroup of some compact group.
A classical result of Comfort and Ross \cite{CR} says that pseudocompact groups are precompact. 

{\em All topological groups in this paper are assumed to be Hausdorff.} 

\section{Introduction}

Recall  that a point $x$ of a topological space $X$ is an {\em accumulation point\/} of a sequence 
$\{x_n:n\in\N\}$ of points of $X$ provided that the set 
$\{n\in\N: x_n\in V\}$ is infinite for every neighbourhood $V$ of $x$ in $X$.

\begin{definition}
\label{strongly:pseudo}
(i) Let $\mathcal{U} = \{U_n:n\in\N\}$ be a sequence of 
sets.
If $x_n\in U_n$ for every $n\in\N$, we shall call the sequence $\{x_n : n \in \N \}$ a \emph{selection for $\mathcal{U}$}.

(ii) A topological space $X$ is {\em selectively pseudocompact\/}
if 
every sequence $\{U_n:n\in\N\}$
of non-empty open subsets of $X$ 
admits
a selection $\{x_n : n \in \N \}$ 
which has an accumulation point in $X$. 
\end{definition}

This
notion was introduced by Garc\'{\i}a-Ferreira and Ortiz-Castillo \cite{GF-OC} under the name ``strong pseudocompactness''.
The present name and an equivalent reformulation of the property given in 
item
(ii) of Definition \ref{strongly:pseudo}
is due to Dorantes-Aldama and the first author \cite[Theorem 2.1 and Definition 2.2]{DS}.
One easily sees that
$$
\text{countably compact}
\to
\text{selectively pseudocompact}
\to
\text{pseudocompact}.
$$

In 2015, 
Garc\'{\i}a-Ferreira and Tomita 
constructed 
a Boolean pseudocompact group 
that is not selectively pseudocompact
\cite[Example 2.4]{GF-T}.
In this paper we establish a fairly general result involving the Bohr topology on all countable subgroups of a given topological group (Theorem \ref{main:result}) from which we deduce that  many known examples in the literature of pseudocompact Boolean groups {\em automatically\/} fail to be selectively pseudocompact.
In particular, each of the strongly self-dual pseudocompact groups constructed by Tkachenko \cite[Theorem 3.3]{T}
in 2009
is not selectively pseudocompact either.
Therefore, even strongly self-dual pseudocompact Boolean groups  need not be
selectively pseudocompact (Corollary \ref{self:dual:corollary}). 
Furthermore, pseudocompact group topologies with property
$\sharp$ 
on Boolean groups
constructed by Galindo and Macario \cite{GM} in 2011
are also not selectively pseudocompact.
Based on this, we show that, under the Singular Cardinal Hypothesis SCH, {\em every\/} pseudocompact Boolean group admits a pseudocompact group topology which fails to be selectively pseudocompact
(Corollary \ref{Galindo:Macario}).

\section{Preliminaries}

Every abelian group $G$ has the maximal precompact
group topology on $G$ called its 
{\em Bohr topology\/}. This topology is simply the initial topology 
with respect to 
the family of all homomorphisms from $G$ to the circle group $\T$.
We shall use $G^{\#}$ to denote the abelian group $G$ endowed with its Bohr topology.

\begin{definition}
\cite{T88}
\label{h-embedded}
A subgroup $H$ of a topological group $G$ is said to be {\em $h$-embedded in $G$\/} if every homomorphism from $H$ to the circle group $\T$ is a restriction of some continuous group homomorphism from $G$ to $\T$.
\end{definition}

The next fact and its corollary are part of folklore.
\begin{fact}
\label{single:sugroup}
If $H$ is an $h$-embedded subgroup of an abelian topological group $G$, then the subspace topology 
on
$H$ induced by the topology of $G$
is finer than
the Bohr topology of $H$.   
\end{fact}

\begin{proof}
Since the Bohr topology of $H$ is the initial topology 
with respect to 
the family of all homomorphisms from $H$ to the circle group $\T$,
the conclusion easily follows from Definition \ref{h-embedded}.
\end{proof}

\begin{corollary}
\label{h:embedded:vs:Bohr}
If $G$ is an abelian topological group such that every countable subgroup of $G$ is $h$-embedded in $G$,  
then the subspace topology 
on each countable subgroup 
$H$ of $G$ 
induced by the topology of $G$
is finer than
the Bohr topology of $H$.   
\end{corollary}

The class of abelian topological groups $G$ such that every countable subgroup of $G$ is $h$-embedded in $G$ plays a prominent role in Pontryagin duality theory, as witnessed by Fact \ref{reflexivity:fact} below.

\begin{remark}
The authors do not know if the converse implications in Fact \ref{single:sugroup}
and Corollary \ref{h:embedded:vs:Bohr} hold.
\end{remark} 

\begin{remark}
\label{h:embedded-remark}
{\em If all countable subgroups of a topological group $G$ are closed in $G$, then the closure of every countable subset of $G$ is countable\/}.
 Indeed, let $X$ be a countable subset of $G$.
Then the subgroup $\grp{X}$ of $G$ is also countable, so it is closed in $G$ by our assumption.
Since $X\subseteq \grp{X}$, the closure of $X$ in $G$ is contained in the countable set $\grp{X}$, so it is countable.
\end{remark}

\begin{fact}
\label{subgroups:of:Bohr:topology:are:closed}
\cite[\S 2]{CS}
If $G$ is an abelian group, then every subgroup $H$ of $G$  is closed in $G^\#$.
\end{fact}

\begin{proposition}
\label{stronger:proposition}
Let $G$ be a topological abelian group such that 
the subspace topology 
on each countable subgroup 
$H$ of $G$ 
induced by the topology of $G$
is finer than
the Bohr topology of $H$.   
Then:
\begin{itemize}
\item[(i)]
  all countable subgroups of $G$
 are closed, and 
\item[(ii)] 
 all
separable pseudocompact subsets of $G$ are finite.
\end{itemize}
\end{proposition}
\begin{proof}
(i)
Let $C$ be a countable subgroup of $G$. To prove that $C$ is closed in $G$, we fix an arbitrary element $g\in G\setminus C$ 
and find an open neighbourhood $V$ of $g$ disjoint from $C$.
Let $H=\grp{C\cup\{g\}}$.
Note that $C$ is closed in $H^\#$ by Fact \ref{subgroups:of:Bohr:topology:are:closed}, so we can fix an open 
neighbourhood $U$ of $g$ in $H^\#$ disjoint from $C$.
Since 
$H$
is countable, by our assumption, the subspace topology on $H$ is finer than the Bohr topology of $H$.
Therefore, we can find an open subset $V$ of $G$ such that 
$V\cap H=U$.
Since $g\in U$ and $U\cap C=\emptyset$, it follows that $V$ is an open neighbourhood of $g$ in $G$ disjoint from $C$.

(ii)
By (i), all countable subgroups of $G$ are closed in $G$, so we can apply Remark \ref{h:embedded-remark} to conclude that the closure of each countable subset of $G$ is countable. 
Let $X$ be in infinite separable pseudocompact subspace of $G$.
By \cite[Lemma 10]{SY}, $X$ contains a non-trivial convergent sequence.  The subgroup $K=\grp{X}$ of $G$ is countable. By the assumption of our proposition, 
the Bohr topology of $K$ is coarser than the subspace topology 
of $K$ inherited from $G$. Since $X$ is a convergent sequence in $K$,  it must be a convergent sequence in $K^\#$.
This contradicts the well-known fact that the Bohr topology does not have non-trivial convergent sequences \cite{Glicksberg} (see 
Remark \ref{Bohr:psc:remark} below
for more general result).
\end{proof}

Combining Corollary \ref{h:embedded:vs:Bohr} with Proposition 
\ref{stronger:proposition}, we obtain the following 

\begin{corollary}
\label{h-emb:weakening}
If all countable subgroups of an abelian group $G$ are $h$-embedded in $G$, then 
 all
separable pseudocompact subsets of $G$ are finite.
\end{corollary}

\begin{remark}
\label{triple:remark}
(1)
Hern\'{a}ndez and Macario~\cite{HM} say that a space $X$ is {\em countably pseudocompact\/} if 
every countable subset of $X$ is contained in a separable pseudocompact subset of $X$.
A space $X$ is said to be {\em countably pracompact\/} if $X$ contains a dense set $Y$ such
that every infinite subset of $Y$ has an accumulation point in $X$; see \cite[Ch. III, Sec. 4]{Arh2}.

(2)
As was noted in the text prior to \cite[Remark 2]{SY}, 
{\em if all separable pseudocompact subsets of a topological space $X$ are finite, then 
$X$ does not 
contain infinite 
subsets which 
are either countably pseudocompact or countably pracompact;
in particular, $G$ does not contain infinite countably compact subsets.}
\end{remark}

\begin{remark}
\label{ACDT:remark}
(a)
In \cite[Proposition 2.1]{ACDT}, 
Ardanza-Trevijano,
Chasco,
Dom\'{\i}nguez
and Tkachenko
proved 
the following result.
{\em Let $G$ be a topological abelian group such that every countable
subgroup of $G$ is $h$-embedded in $G$. Then all countable subgroups of $G$
are closed, all
compact subsets of $G$ are finite, and $G$ is sequentially closed in its completion.\/}

(b) Corollary \ref{h:embedded:vs:Bohr} shows that the assumption of Proposition \ref{stronger:proposition} is (potentially) weaker than that of the statement in item (a).

(c)
It follows from Remark \ref{triple:remark}~(2) that
item (ii) of Proposition \ref{stronger:proposition}
strengthens the second conclusion of the statement quoted in (a).
In particular, 
Corollary \ref{h-emb:weakening}
strengthens the second conclusion of 
the statement in item (a).
\end{remark}

\begin{remark}
\label{Bohr:psc:remark}
Item (ii) of Proposition \ref{stronger:proposition} should be compared with \cite[Corollary 6.4]{DS-Kronecker} which says that {\em the Bohr topology of an abelian group does not have infinite pseudocompact subsets\/}.
Indeed,
it is easy to see that every subgroup $H$ of the topological group $G^\#$ 
is $h$-embedded in $G^\#$ and inherits from $G^\#$ its Bohr topology. This means that every abelian group $G$ equipped with its Bohr topology satisfies the assumptions of Proposition \ref{stronger:proposition} and its 
Corollary \ref{h-emb:weakening}.
Therefore, the non-existence of infinite {\em separable\/}
pseudocompact subsets in 
item (ii) of Proposition \ref{stronger:proposition}
can be viewed as the best possible conclusion under 
the much weaker assumptions of Proposition \ref{stronger:proposition} and 
Corollary \ref{h-emb:weakening}.
That the word ``separable'' cannot be omitted from item (ii) of Proposition \ref{stronger:proposition} and its Corollary \ref{h-emb:weakening} will be seen from numerous examples of pseudocompact groups satisfying
the assumptions of these two results exhibited in Section 
\ref{section:applications}.
\end{remark}

A subset $X$ of an abelian group $G$ is {\em independent\/}
if $0\not\in X$ and $\grp{A}\cap \grp{X\setminus A}=\{0\}$ for every subset $A$ of $X$.

We shall need
the following result 
of 
Hart and van Mill shown in \cite[Lemma 1.4]{H-vM}.

\begin{fact}
\label{hart-vanmill:lemma}
Every independent subset 
of an abelian group $G$
is closed and discrete in $G^{\#}$.
\end{fact}

A straightforward proof of the following folklore fact is left to the reader.
\begin{fact}
\label{independence:for:Boolean:groups}
A subset $X$ of a Boolean group 
is independent if and only if 
$\sum_{a\in A} a\not=0$ for every non-empty finite subset $A$ of $X$.
\end{fact}

\section{An auxiliary construction}

This section is inspired by the proof of \cite[Lemma 2.1]{GF-OC}.

Even though 
all results in this section 
hold
for arbitrary (not necessarily commutative) topological groups $G$, we shall use abelian notations denoting the group operation of $G$ by $+$ and its identity element by~$0$.

\begin{lemma}
\label{inductive:construction}
Let $G$ be a non-discrete topological group
and $V_{-1}=G$. Then there exist a sequence of points $\{g_n : n \in \N \} \subseteq G \setminus \{0\}$
and a sequence $\{ V_n : n \in \N \}$ of open neighbourhoods of $0$ such that 
the following conditions hold
for every $n\in\N$:
\begin{itemize}
\item [($\alpha_n$)]
$0 \not \in g_n + V_n + V_n\subseteq V_{n-1}$, 
\item[($\beta_n$)] $V_n\subseteq V_{n-1}$.
\end{itemize}

\end{lemma}
\begin{proof}
By induction on $n\in\N$, we shall select $g_n\in G\setminus\{0\}$
and an open neighbourhood $V_n$ of $0$ satisfying ($\alpha_n$).

Fix $n\in\N$. Suppose 
that  $g_i\in G_i\setminus\{0\}$ and an open neighbourhood $V_i$ of $0$ satisfying  ($\alpha_i$) 
and ($\beta_i$)
have already been chosen for every $i=0,\dots,n-1$. 
Since $V_{n-1}$ is an open neighbourhood of $0$
and $G$ is
non-discrete, we have $V_{n-1}\not=\{0\}$.
Then $U=V_{n-1}\setminus\{0\}$ is a non-empty open 
subset of $G$.
Thus, we can select some  $g_n\in U$.
Using the continuity of the group operation of $G$, we can find
an open neighbourhood $V_n$ of $0$ such that
$g_n+V_n+V_n\subseteq U$
and
$V_n\subseteq V_{n-1}$.
Now ($\alpha_n$) 
and ($\beta_n$) hold.
\end{proof}

\begin{lemma}
\label{lemma:2.2}
Under the assumptions of Lemma \ref{inductive:construction},
\begin{equation}
\label{eq:1:k}
\sum_{j=1}^k (g_{i_j} + V_{i_j}) + V_{i_k} \subseteq g_{i_1} + V_{i_1} + V_{i_1}
\end{equation}
for every strictly increasing finite sequence $i_1, \dots, i_k \in \N$.
\end{lemma}

\begin{proof}
We prove this lemma by induction on the length $k$ of the sequence. Note that \eqref{eq:1:k}
trivially holds for sequences of length $1$.
Suppose that
\eqref{eq:1:k} has been verified for all strictly increasing sequences of length $k$, 
and let $i_1, \dots, i_{k+1} \in \N$ be a 
strictly increasing
sequence. Observe 
that
\begin{equation}
\label{eq:2:b}
\sum_{j=1}^{k+1} (g_{i_j} + V_{i_j}) + V_{i_{k+1}} \subseteq \sum_{j=1}^{k} (g_{i_j} + V_{i_j}) + g_{i_{k+1}} + V_{i_{k+1}} + V_{i_{k+1}} \subseteq \sum_{j=1}^{k} (g_{i_j} + V_{i_j}) 
+ V_{i_{k+1}-1}
\end{equation}
by 
($\alpha_{i_{k+1}}$).
Since $i_k<i_{k+1}$, from ($\beta_{i_k+1}$), ($\beta_{i_k+2}),\dots,(\beta_{i_{k+1}-1})$ we obtain
\begin{equation}
\label{nested:Vs}
V_{i_k}\supseteq V_{i_k+1}\supseteq\dots\supseteq V_{i_{k+1}-1}.
\end{equation}
By inductive hypothesis, \eqref{eq:1:k} holds.
Combining 
\eqref{eq:1:k}, \eqref{eq:2:b} and \eqref{nested:Vs},
we get
$$\sum_{j=1}^{k+1} (g_{i_j} + V_{i_j}) + V_{i_{k+1}} 
\subseteq \sum_{j=1}^{k} (g_{i_j} + V_{i_j}) 
+ V_{i_{k+1}-1}
\subseteq \sum_{j=1}^{k} (g_{i_j} + V_{i_j}) + V_{i_k} \subseteq g_{i_1} + V_{i_1} + V_{i_1}.$$
This finishes the inductive step.
\end{proof}

\begin{corollary}
\label{garcia-tomita:lemma}
Each non-discrete group $G$ has
a sequence 
$\{ U_n : n \in \N \}$ 
of non-empty open subsets of $G$ such that 
\begin{equation}
\label{zero:is:not:in:the:sum}
0\not\in \sum_{j=1}^{k} U_{i_j}
\end{equation}
for every 
strictly increasing finite sequence $i_1, \dots, i_k \in \N$.
\end{corollary}

\begin{proof}
Consider the sequences 
$\{g_n: n \in \N\}$ and $\{V_n : n \in \N \}$ 
as in the conclusion of 
Lemma \ref{inductive:construction}.
Let 
$U_n = g_n + V_n$ for every $n \in \N$. 
Clearly, $\mathcal{U}=\{U_n : n \in \N \}$ is a sequence of non-empty open subsets of $G$.

Let $i_1, \dots, i_k \in \N$ be a strictly increasing finite sequence. 
Then
$$
\sum_{j=1}^{k} U_{i_j} = \sum_{j=1}^{k} (g_{i_j} + V_{i_j}) 
\subseteq 
\sum_{j=1}^{k} (g_{i_j} + V_{i_j}) 
+
V_{i_k}
\subseteq g_{i_1} + V_{i_1} + V_{i_1}
$$
by $0\in V_{i_k}$ and
Lemma \ref{lemma:2.2}.
Since $0\not\in g_{i_1} + V_{i_1} + V_{i_1}$ by 
($\alpha_{i_{1}}$),
this implies \eqref{zero:is:not:in:the:sum}.
\end{proof}

\section{Main result}

We are now ready to state and prove our main result.

\begin{theorem}
\label{main:result}
Let $G$ be 
an infinite
Boolean 
topological group such that 
the subspace topology 
on
every countable subgroup $H$ of $G$
is finer than
the Bohr topology of $H$.
Then $G$ is not selectively pseudocompact.
\end{theorem}

\begin{proof}
Suppose that $G$ is discrete.
Since $G$ is 
infinite, it is not pseudocompact,  and thus $G$ is not selectively pseudocompact either.

From now on we shall suppose that $G$ is non-discrete.
Let $\mathcal{U} = \{U_n : n \in \N \}$ be the 
sequence
of 
non-empty open subsets of $G$ as in 
Corollary~\ref{garcia-tomita:lemma},
and let $\{x_n : n \in \N \}$ be an arbitrary selection for $\mathcal{U}$. 
According to item (ii) of Definition \ref{strongly:pseudo},
in order 
to
prove that $G$ is not selectively pseudocompact, it suffices to show that 
the sequence
$\{x_n : n \in \N \}$ 
does not have 
an accumulation point 
in $G$.
This will be established in Claim~\ref{new:claim} below.

\begin{claim}
\label{claim:1}
$X=\{x_n:n\in\N\}$ is a faithfully indexed independent subset of $G$.
\end{claim}
\begin{proof}
First, we check that $X$ is faithfully indexed. Let $m,n\in\N$ and $m<n$. Then 
$0\not\in U_m+U_n$
by Corollary~\ref{garcia-tomita:lemma}.
On the other hand,
$x_m+x_n\in U_m+U_n$, so 
$x_m+x_n\not=0$. Since $G$ is a Boolean group, this implies $x_m\not=x_n$.

Next, we check that $X$ is independent.
By Fact \ref{independence:for:Boolean:groups},
it suffices to show that 
$\sum_{j=1}^k x_{i_j}\not=0$ for every faithfully indexed finite subset $\{i_j:j=1,\dots,k\}$ of $\N$. Without loss of generality, we may assume that the sequence $i_1,\dots,i_k$ is 
strictly
increasing. Since
$\sum_{j=1}^k x_{i_j}\in \sum_{j=1}^k U_{i_j}$, from 
Corollary \ref{garcia-tomita:lemma} we conclude that $\sum_{j=1}^k x_{i_j}\not=0$.
\end{proof}

\begin{claim}
\label{new:claim}
The sequence
$\{x_n:n\in\N\}$ does not have 
an accumulation point in $G$.
\end{claim}

\begin{proof}
Let $g\in G$ be arbitrary,
and let $X$ be the set as in Claim~\ref{claim:1}.
The
subgroup $H=\grp{X \cup \{g\}}$ 
of $G$
is countable. 
Since $X$ is an independent subset of $G$ by Claim \ref{claim:1},
it is also an independent subset of $H$, as $H$ is a subgroup of $G$ containing $X$.
By Fact \ref{hart-vanmill:lemma}, $X$ is closed and discrete in $H^{\#}$.
Since $g\in H$, we can find an open neighbourhood $U$ of $g$ in $H^\#$
such that $U\cap X\subseteq \{g\}$.
Since the subspace topology on $H$ 
is finer than 
its Bohr topology by our assumption on $G$, 
the set $U$ is open in $H$.
Thus, 
there exists an open subset $V$ of $G$ such that $V \cap H = U$. Since $X\subseteq H$, we get 
$V\cap X=V\cap H\cap X=U\cap X\subseteq \{g\}$.
Since the set $X$ is faithfully indexed by Claim \ref{claim:1},
this means that $g$ cannot be 
an accumulation point
of the sequence
$\{x_n:n\in\N\}$.
Since $g\in G$ was taken arbitrarily, this means that the sequence
$\{x_n:n\in\N\}$ does not have 
an accumulation point in $G$.
\end{proof}
\end{proof}

From Corollary \ref{h:embedded:vs:Bohr}
and 
Theorem \ref{main:result}, one obtains the following

\begin{corollary}
\label{precompact:corollary}
Let $G$ be 
an infinite
Boolean group all countable subgroups of which are $h$-embedded in $G$. Then $G$ is not selectively pseudocompact.
\end{corollary}

The rest of the paper is devoted to applications of 
Corollary \ref{precompact:corollary}
aimed at establishing the abundance of pseudocompact abelian groups which are not selectively pseudocompact.

\section{Applications of the main result}
\label{section:applications}

For an abelian topological group $G$, we denote by $\dual{G}$ the group of all continuous group homomorphisms  from $G$ to $\T$ endowed with the compact-open topology. Recall that $\dual{G}$  is called the {\em Pontryagin dual\/} of $G$.
For each $g\in G$, the map $\psi_g:\dual{G}\to\T$ defined by
$\psi_g(h)=h(g)$ for every $h\in\dual{G}$, is a continuous homomorphism from $\dual{G}$ to $\T$, so
$\psi_g\in \seconddual{G}$; that is, $\psi_g$ is an element of the second dual $\seconddual{G}$ of $G$.
Consider the map $\alpha_G: G\to \seconddual{G}$ defined 
by $\alpha_G(g)=\psi_g$ for $g\in G$.
The group $G$ is called {\em (Pontryagin) reflexive\/}
when $\alpha_G$ is a topological isomorphism between $G$ and $\seconddual{G}$.

The following result is very useful 
for constructing examples of reflexive groups. 
It follows from
\cite[Lemma 2.3 and Theorem 6.1]{GM}
and can be found explicitly
in \cite[Theorem 2.8]{ACDT}.

\begin{fact}
\label{reflexivity:fact}
If all countable subgroups of a pseudocompact abelian group $G$ are $h$-embedded in $G$, then $G$ is reflexive.
\end{fact}

The
following result of Galindo and Macario attests to an abundance of pseudocompact reflexive abelian groups $G$ having the property that all  countable subgroups of $G$ are $h$-embedded in $G$.

\begin{fact}
\label{fact:Galindo:Macario}
Let $G$ be an infinite
pseudocompact abelian
group. Assume also that one of the following two conditions is satisfied:
\begin{itemize}
\item[(i)] 
$|G|\le 2^{2^\cont}$;
\item[(ii)] 
the Singular Cardinal Hypothesis SCH holds. 
\end{itemize}
Then $G$ admits a pseudocompact 
group topology $\tau$ 
such that every countable subgroup of $(G,\tau)$ is $h$-embedded in it.
In particular, $(G,\tau)$ is a reflexive group and all compact subsets of $(G,\tau)$ are finite.
\end{fact}
\begin{proof}
By \cite[Corollary 5.6 and Theorem 5.8]{GM}, $G$ admits a pseudocompact group topology $\tau$ such that every countable subgroup of $(G,\tau)$ is $h$-embedded in it (topological groups with this property are said to have property $\sharp$ in \cite[Definition 2.2]{GM}). The reflexivity of $(G,\tau)$ follows from
Fact \ref{reflexivity:fact}.
By Remark \ref{ACDT:remark}~(a),
all compact subsets of $(G,\tau)$ are finite.
\end{proof}

\begin{remark}
\label{remark:no:infinite:separable:pseudocompacts}
It follows from 
Corollary \ref{h-emb:weakening} and Remark \ref{triple:remark}~(2)
that 
the topology $\tau$ in Fact \ref{fact:Galindo:Macario} has the stronger property that all separable pseudocompact subsets of $(G,\tau)$ are finite.
\end{remark}

\begin{corollary}
\label{Galindo:Macario}
An infinite
pseudocompact Boolean group
$G$ admits a pseudocompact reflexive group topology which is not selectively pseudocompact in each of the following 
cases:
\begin{itemize}
\item[(i)] 
$|G|\le 2^{2^\cont}$;
\item[(ii)] 
the Singular Cardinal Hypothesis SCH holds. 
\end{itemize}
\end{corollary}
\begin{proof}
Let $G$ be an infinite pseudocompact Boolean group.
Apply Fact \ref{fact:Galindo:Macario} to find
a
 pseudocompact reflexive group topology  $\tau$ on $G$
such that every countable subgroup of $(G,\tau)$ is $h$-embedded in it.
Finally,
$(G,\tau)$ is not selectively pseudocompact by Corollary \ref{precompact:corollary}.
\end{proof}

It follows from Corollary \ref{Galindo:Macario} that, under the assumption of SCH,
{\em every\/} pseudocompact Boolean group $G$ can be equipped with 
a (reflexive) pseudocompact group topology which fails to be selectively pseudocompact. Moreover, for ``small'' groups $G$ 
the assumption of SCH is superfluous.

Note that SCH is a ``rather mild'' additional set-theoretic assumption beyond ZFC, the Zermelo-Fraenkel axioms of set theory augmented by the Axiom of Choice. Indeed, the failure of SCH implies the existence of 
a large cardinal \cite{DJ}.
Nevertheless, 
we do not know if one can omit the Singular Cardinal Hypothesis SCH completely; see Question \ref{question}.

An abelian topological group $G$ is called {\em self-dual\/} if $G$ is topologically isomorphic to its Pontryagin dual $\dual{G}$.

\begin{corollary}
\label{self-dual-corollary}
An infinite
self-dual pseudocompact Boolean group cannot be selectively pseudocompact.
\end{corollary}
\begin{proof}
Let $G$ be a self-dual pseudocompact Boolean group.
By \cite[Theorem 2.3]{T}, all countable subgroups of $G$ are $h$-embedded in $G$, so
$G$ is not selectively pseudocompact by Corollary \ref{precompact:corollary}. 
\end{proof}

We refer the reader to \cite{T} for the definition of strong self-duality. As can be guessed from the terminology, strongly self-dual abelian topological groups are self-dual. 
By \cite[Proposition 2.2]{T}, strongly self-dual abelian topological groups are reflexive.

Finally, we recall the
result of Tkachenko shown in \cite[Theorem 3.3]{T}.

\begin{fact}
\label{Tkachenko:Example}
Let $\kappa$ be an infinite cardinal with $\kappa^\omega = \kappa$. Then
there exists a pseudocompact strongly self-dual Boolean group $G$ satisfying $|G| = w(G) = \kappa.$ 
\end{fact}

\begin{corollary}
For every infinite cardinal $\kappa$ satisfying $\kappa^\omega = \kappa$, there exists a pseudocompact non-selectively pseudocompact strongly self-dual Boolean group $G$ such that $|G| = w(G) = \kappa.$
\end{corollary}

\begin{proof}
Let $\kappa$ be an infinite cardinal such that $\kappa^\omega = \kappa$. Apply Fact \ref{Tkachenko:Example} to find 
a pseudocompact strongly self-dual Boolean group $G$ satisfying $|G| = w(G) = \kappa.$
Then $G$ is self-dual,
so $G$ cannot be selectively pseudocompact by Corollary \ref{self-dual-corollary}.
\end{proof}

\begin{corollary}
\label{self:dual:corollary}
A strongly self-dual pseudocompact Boolean group need not be selectively pseudocompact.
\end{corollary}

\section{Open questions}

It follows from Fact \ref{fact:Galindo:Macario}~(ii) and Remark 
\ref{remark:no:infinite:separable:pseudocompacts}
that, under the assumption of SCH, every infinite pseudocompact abelian group $G$  
admits a pseudocompact group topology $\tau$ such that 
all separable pseudocompact subsets of $(G,\tau)$ are finite.
It is not clear whether this result can be proved without any additional set-theoretic assumptions beyond ZFC.

\begin{question}
Is it true in ZFC that every infinite pseudocompact abelian group $G$ admits a pseudocompact group topology $\tau$ such that 
all separable pseudocompact subsets of $(G,\tau)$ are finite? 
\end{question}
It follows from Fact \ref{fact:Galindo:Macario}~(i) that the answer to this question is positive when $|G|\le 2^{2^\cont}$.

\begin{question}
\label{question}
Is it true in ZFC that every pseudocompact Boolean group admits
a pseudocompact non-selectively pseudocompact group topology?
Can this topology be made also reflexive? 
\end{question}

As was noted after Corollary \ref{Galindo:Macario}, the answer to this question is positive under SCH.

It is unclear whether Corollary \ref{Galindo:Macario} holds beyond the class of Boolean groups.

\begin{question}
Does every pseudocompact abelian group admit a pseudocompact group topology which is not selectively pseudocompact?
Can this topology be made also reflexive? 
\end{question}

\end{document}